\documentclass[12pt]{amsart}
\usepackage{amssymb, amstext, amscd, amsmath}
\usepackage{mathtools, xypic, color, dsfont}

\makeatletter
\def\@cite#1#2{{\m@th\upshape\bfseries%
[{#1\if@tempswa{\m@th\upshape\mdseries, #2}\fi}]}}
\makeatother

\mathtoolsset{centercolon}

\theoremstyle{plain}
\newtheorem{thm}{Theorem}[section]
\newtheorem{cor}[thm]{Corollary}

\newtheorem{lem}[thm]{Lemma}

\theoremstyle{definition}

\newtheorem{defn}[thm]{Definition}

\newcommand{\bD}{{\mathds{D}}}

\newcommand{\bN}{{\mathds{N}}}

\newcommand{\bT}{{\mathds{T}}}
\newcommand{\bZ}{{\mathds{Z}}}

  \newcommand{\A}{{\mathcal{A}}}
  \newcommand{\B}{{\mathcal{B}}}
  \newcommand{\C}{{\mathcal{C}}}

\renewcommand{\H}{{\mathcal{H}}}
  
  \newcommand{\J}{{\mathcal{J}}}
  \newcommand{\K}{{\mathcal{K}}}

\renewcommand{\P}{{\mathcal{P}}}

\renewcommand{\S}{{\mathcal{S}}}

\renewcommand{\phi}{\varphi}
\newcommand{\upchi}{{\raise.35ex\hbox{\ensuremath{\chi}}}}

\newcommand{\fA}{{\mathfrak{A}}}
\newcommand{\fB}{{\mathfrak{B}}}

\newcommand{\ft}{{\mathfrak{t}}}
\newcommand{\fu}{{\mathfrak{u}}}

\newcommand{\rC}{{\mathrm{C}}}

\newcommand{\qand}{\quad\text{and}\quad}

\newcommand{\qforal}{\quad\text{for all}\ }

\newcommand{\AND}{\text{ and }}

\newcommand{\AD}{\mathrm{A}(\mathds{D})}
\newcommand{\CT}{{\mathrm{C}(\mathds{T})}}

\newcommand{\cenv}{\mathrm{C}^*_{\text{e}}}

\newcommand{\gal}{\alpha_{\lambda}}

\newcommand{\ol}{\overline}

\newenvironment{sbmatrix}{\left[\begin{smallmatrix}}{\end{smallmatrix}\right]}

\begin{document}
\title{Semicrossed products of the disk algebra}

\author[K.R.Davidson]{Kenneth R. Davidson}
\thanks{First author partially supported by an NSERC grant.}
\address{Pure Mathematics Department\\
University of Waterloo\\Waterloo, ON\; N2L--3G1\\CANADA}
\email{krdavids@uwaterloo.ca}

\author[E.G. Katsoulis]{Elias~G.~Katsoulis}
\thanks{Second author was partially supported by a grant from ECU}
\address{ Department of Mathematics\\
University of Athens\\
15784 Athens \\GREECE \vspace{-2ex}}
\address{\textit{Alternate address:} Department of Mathematics\\
East Carolina University\\ Greenville, NC 27858\\USA}
\email{katsoulise@ecu.edu}

\begin{abstract}
If $\alpha$ is the endomorphism of the disk algebra, $\AD$, induced by
composition with a finite Blaschke product $b$, then the semicrossed product
$\AD\times_{\alpha} \bZ^+$ imbeds canonically, completely isometrically
into $\rC(\bT)\times_{\alpha} \bZ^+$.  Hence in the case of a non-constant Blaschke product $b$, the C*-envelope has the form
$ \rC(\S_{b})\times_{s} \bZ$,
where $(\S_{b}, s)$ is the solenoid system for $(\bT, b)$. In the case where $b$ is a constant, then the C*-envelope of $\AD\times_{\alpha} \bZ^+$ is strongly Morita equivalent to a crossed product of the form $ \rC(\S_{e})\times_{s} \bZ$, where  $e \colon \bT \times \bN \longrightarrow \bT \times \bN$ is a suitable map and $(\S_{e}, s)$ is the solenoid system for $(\bT \times \bN , \, e)$ .
\end{abstract}


\date{}
\maketitle

\section{Introduction} \label{S:intro}

If $\A$ is a unital operator algebra and $\alpha$ is a completely contractive
endomorphism, the semicrossed product is an operator algebra $\A \times_\alpha \bZ_+$
which encodes the covariant representations of $(\A,\alpha)$: namely completely contractive
unital representations $\rho:\A\to\B(\H)$ and contractions $T$ satisfying
\[ \rho(a) T = T \rho(\alpha(a)) \qforal a \in \A .\]
Such algebras were defined by Peters \cite{Pet} when $\A$ is a C*-algebra.

One can readily extend Peter's definition \cite{Pet} of the
semicrossed product of a C*-algebra by a $*$-endomorphism to unital operator
algebras and unital completely contractive endomorphisms.
One forms the \textit{polynomial algebra} $\P(\A,\ft)$ of formal polynomials
of the form $p=\sum_{i=0}^n \ft^i a_i$, where $a_i\in\A$, with multiplication
determined by the covariance relation $a\ft = \ft \alpha(a)$ and the norm
\[ \| p \| = \sup_{(\rho,T) \text{ covariant}} \big\| \sum_{i=0}^n T^i \rho(a_i) \big\| .\]
This supremum is clearly dominated by $\sum_{i=0}^n \|a_i\|$; so this norm
is well defined.
The completion is the semicrossed product $\A \times_\alpha \bZ_+$.
Since this is the supremum of operator algebra norms, it is also
an operator algebra norm.
By construction, for each covariant representation $(\rho,T)$,
there is a unique completely contractive representation
$\rho \times T$ of $\A \times_\alpha \bZ_+$ into $\B(\H)$ given by
\[ \rho \times T(p) = \sum_{i=0}^nT^i \rho(a_i) .\]
This is the defining property of the semicrossed product.

In this note, we examine semicrossed products of the disk algebra by
an endomorphism which extends to a $*$-endomorphism of $\CT$.
In the case where the endomorphism is injective, these have the
form $\alpha(f) = f\circ b$ where $b$ is a non-constant Blaschke product.
We show that every covariant representation of $(\AD,\alpha)$ dilates to a covariant
representation of $(\CT,\alpha)$.
This is readily dilated to a covariant representation $(\sigma,V)$,
where $\sigma$ is a $*$-representation of $\CT$ (so $\sigma(z)$ is unitary)
and $V$ is an isometry.
To go further, we use the recent work of Kakariadis and Katsoulis \cite{KK}
to show that $\CT \times _\alpha \bZ^+$ imbeds completely isometrically into
a C*-crossed product $\rC(\S_b) \times_\alpha \bZ$.
In fact, $\cenv(\CT \times _\alpha \bZ^+)= \rC(\S_b) \times_\alpha \bZ$ and as a consequence, we obtain that $(\rho,T)$ dilates to a covariant representation
$(\tau,W)$, where $\tau$ is a $*$-representation of $\CT$ (so $\sigma(z)$ is unitary)
and $W$ is a unitary.

In contrast, if $\alpha$ is induced by a constant Blashcke product,
we can no longer identify $\cenv(\CT \times _\alpha \bZ^+)$ up to isomorphism.
In that case, $\alpha$ is evaluation at a boundary point. Even though every covariant representation of $(\AD,\alpha)$ dilates to a covariant
representation of $(\CT,\alpha)$, the theory of \cite{KK} is not directly applicable
since $\alpha$ is not injective. Instead, we use the process of
``adding tails to C*-correspondences" \cite{MT}, as modified in
\cite{DR, KK2} and we identify $\cenv(\CT \times _\alpha \bZ^+)$
up to strong Morita equivalence as a crossed product.
In Theorem~\ref{constantcase} we show that
$\cenv(\CT \times _\alpha \bZ^+)$ is strongly Morita equivalent to a 
C*-algebra of the form $ \rC(\S_{e})\times_{s} \bZ$,
where  $e \colon \bT \times \bN \longrightarrow \bT \times \bN$ is a
suitable map and $(\S_{e}, s)$ is the solenoid system for $(\bT \times \bN , \, e)$.

Semi-crossed products of the the disc algebra were introduced and first studied 
by Buske and Peters in \cite{BP}, following relevant work of 
Hoover, Peters and Wogen \cite{HPW}. 
The algebras $\AD\times_{\alpha} \bZ^+$, where $\alpha$ is an arbitrary 
endomorphism, where classified up to algebraic endomorphism in \cite{DKconj}. 
Results associated with their C*-envelope can be found in 
\cite[Proposition III.13]{BP} and \cite[Thoorem 2]{Pow}. 
The results of the present paper subsume and extend these earlier results.

\section{The Disk Algebra} \label{S:covariance}

The C*-envelope of the disk algebra $\AD$ is $\CT$, the space of continuous
functions on the unit circle.
Suppose that $\alpha$ is an endomorphism of $\CT$ which leaves $\AD$ invariant.
We refer to the restriction of $\alpha$ to $\AD$ as $\alpha$ as well.
Then $b = \alpha(z) \in \AD$; and has spectrum
\[
 \sigma_{\AD}(b) \subset \sigma_{\AD}(z) = \ol{\bD}
 \qand \sigma_{\CT}(b) \subset \sigma_{\CT}(z) = \bT.
\]
Thus $\|b\|=1$ and $b(\bT) \subset\bT$.
It follows that $b$ is a finite Blaschke product.
Therefore $\alpha(f) = f \circ b$ for all $f \in \CT$.
When $b$ is not constant, $\alpha$ is completely isometric.

A (completely) contractive representation $\rho$ of $\AD$ is determined
by $\rho(z) = A$, which must be a contraction.  The converse follows
from the matrix von Neumann inequality; and shows that $\rho(f) = f(A)$ is
a complete contraction.
A covariant representation of $(\AD,\alpha)$ is thus
determined by a pair of contractions $(A,T)$ such that $AT=T b(A)$.
The representation of $\AD \times_{\alpha} \bZ^+$ is given by
\[  \rho\times T \big( \sum_{i=0}^n \ft^i f_i \big) = \sum_{i=0}^n T^i f_i(A) ,\]
which extends to a completely contractive representation of the semicrossed
product by the universal property.

A contractive representation $\sigma$ of $\CT$ is a $*$-representation,
and is likewise determined by $U = \sigma(z)$, which must be unitary;
and all unitary operators yield such a representation by the functional calculus.
A covariant representation of $(\CT,\alpha)$ is given by a pair $(U,T)$
where $U$ is unitary and $T$ is a contraction satisfying $UT=T b(U)$.
To see this,  multiply on the left by $U^*$ and on the right by $b(U)^*$
to obtain the identity
\[ U^*T = Tb(U)^* = T \bar{b}(U) = T \alpha(\bar{z})(U) .\]
The set of functions $\{f \in \CT : f(U) T = T \alpha(f)(U) \}$ is easily seen to
be a norm closed algebra.  Since it contains $z$ and $\bar z$, it is all
of $\CT$.  So the covariance relation holds.

\begin{thm} \label{basic}
Let $b$ be a finite Blaschke product, and let $\alpha(f) = f \circ b$.
Then $\AD \times_{\alpha} \bZ^+$ is $($canonically completely isometrically isomorphic to$)$
a subalgebra of $\CT \times_\alpha \bZ^+$.
\end{thm}

\begin{proof}
To establish that $\AD \times_{\alpha} \bZ^+$ is completely isometric to a
subalgebra of $\CT \times_\alpha \bZ^+$,  it suffices to show that each
$(A,T)$ with $AT=Tb(A)$ has a dilation to a pair $(U,S)$ with $U$ unitary
and $S$ a contraction such that $US=Sb(U)$
and $P_\H S^nU^m|_\H = T^nA^m$ for all $m,n \ge 0$.
This latter condition is equivalent to $\H$ being semi-invariant for the
algebra generated by $U$ and $S$.

The covariance relation can be restated as
\[
 \begin{bmatrix}A&0\\0&b(A)\end{bmatrix}
 \begin{bmatrix}0&T\\0&0\end{bmatrix} =
 \begin{bmatrix}0&T\\0&0\end{bmatrix}
 \begin{bmatrix}A&0\\0&b(A)\end{bmatrix}
\]
Dilate $A$ to a unitary $U$ which leaves $\H$ semi-invariant.
Then $\begin{sbmatrix}A&0\\0&b(A)\end{sbmatrix}$ dilates to
$\begin{sbmatrix}U&0\\0&b(U)\end{sbmatrix}$.
By the Sz.Nagy-Foia\c s\ Commutant Lifting Theorem, we may dilate
$\begin{sbmatrix}0&T\\0&0\end{sbmatrix}$ to a contraction of the form
$\begin{sbmatrix}*&S\\ *& *\end{sbmatrix}$ which commutes with
$\begin{sbmatrix}U&0\\0&\alpha(U)\end{sbmatrix}$ and has $\H\oplus\H$
as a common semi-invariant subspace.
Clearly, we may take the $*$ entries to all equal $0$ without changing things.
So $(U,S)$ satisfies the same covariance relations
$US = Sb(U)$.
Therefore we have obtained a dilation to the covariance relations for $(\CT,\alpha)$.
\end{proof}

Once we have a covariance relation for  $(\CT,\alpha)$,
we can try to dilate further.
Extending $S$ to an isometry $V$ follows a well-known path.
Observe that
\[ b(U)S^*S = S^*US = S^*S b(U) .\]
Thus $D = (I-S^*S)^{1/2}$ commutes with $b(U)$.
Write $b^{(n)}$ for the composition of $b$ with itself $n$ times,
Hence we can now use the standard Schaeffer dilation of $S$ to an isometry
$V$ and simultaneously dilate $U$ to $U_1$ as follows:
\[
 V =
 \begin{bmatrix} S & 0 & 0 & 0 & \dots\\
 D & 0 & 0 & 0 & \dots\\
 0 & I & 0 & 0 & \dots\\
 0 & 0 & I & 0 & \dots\\
 \vdots & \vdots & \vdots & \ddots & \ddots
 \end{bmatrix}
 \AND
 U_1 =
 \begin{bmatrix} U & 0 & 0 & 0 & \dots\\
 0 & b(U_1)\! & 0 & 0 & \dots\\
 0 & 0 & \!\!b^{(2)}(U_1)\!\! & 0 & \dots\\
 0 & 0 & 0 &\!\! b^{(3)}(U_1)\!\! & \dots\\
 \vdots & \vdots & \vdots & \vdots & \ddots
 \end{bmatrix} .
\]
A simple calculation shows that $U_1 V = V b(U_1)$.
So as above, $(U,V)$ satisfies the covariance relations for $(\CT,\alpha)$.
\bigbreak

We would like to make $V$ a unitary as well.  This is possible in the case where $b$ is non-constant,
but the explicit construction is not obvious. Instead, we use the theory of C*-envelopes and maximal dilations. First we need the following.

\begin{lem} \label{envelopeidentified}
Let $b$ be a finite Blaschke product, and let $\alpha(f) = f \circ b$. Then
\[
 \cenv(\AD \times_{\alpha} \bZ^+) \simeq \cenv(\CT \times_{\alpha} \bZ^+) .
\]
\end{lem}

\begin{proof} The previous Theorem identifies $\AD \times_{\alpha} \bZ^+$
completely isometrically as a subalgebra of $\CT \times_{\alpha} \bZ^+$.
The C*-envelope $\C$ of $\CT \times_{\alpha} \bZ^+$ is a
Cuntz-Pimsner algebra containing a copy of $\CT$
which is invariant under gauge actions.
Now $\C$ is a C*-cover of $\CT \times_{\alpha} \bZ^+$, so it is easy
to see that it is also a C*-cover of $\AD \times_{\alpha} \bZ^+$.
Since $\AD \times_{\alpha} \bZ^+$ is invariant under the same gauge actions,
its Shilov ideal $\J \subseteq \C$ will be invariant by these actions as well.
If $\J \neq 0 $ then by gauge invariance $\J \bigcap \CT \neq 0$.
Since the quotient map
\[
\AD \longrightarrow \CT \slash (\J\cap \CT)
\]
is completely isometric, we obtain a contradiction. Hence $\J = 0$ and the conclusion follows.
\end{proof}

We now recall some of the theory of semicrossed products of C*-algebras.
When $\fA$ is a C*-algebra, the completely isometric endomorphisms are the
faithful $*$-endomorphisms. In this case,
Peters shows \cite[Prop.I.8]{Pet} that there is a unique C*-algebra $\fB$,
a $*$-automorphism $\beta$ of $\fB$ and an injection $j$ of $\fA$ into $\fB$
so that $\beta\circ j = j \alpha$ and $\fB$ is the closure
of $\bigcup_{n\ge0} \beta^{-n}(j(\fA))$.
It follows \cite[Prop.II.4]{Pet} that $\fA \times_\alpha \bZ_+$ is
completely isometrically isomorphic to the subalgebra of the crossed product algebra
$\fB \times_\beta \bZ$ generated as a non-self-adjoint algebra by an isomorphic
copy $j(\fA)$ of $\fA$ and the unitary $\fu$ implementing $\beta$ in the crossed product.
Actually, Kakariadis and the second author \cite[Thm.2.5]{KK} show that $\fB \times_\beta \bZ$
is the C*-envelope of $\fA \times_\alpha \bZ_+$.

In the case where $\fA= \rC (X)$ is commutative and $\alpha$ is induced by an injective self-map of $X$, the pair $(\fB, \beta)$ has an alternative description.

\begin{defn}
Let $X$ be a Hausdorff space and $\phi$ a surjective self-map of $X$. We define the \textit{solenoid system of }$(X, \phi)$ to be the pair $(\S_{\phi}, s)$, where
\[
\S_{\phi} = \{ (x_n)_{n\ge1 } : x_n = \phi(x_{n+1}), x_n \in X, n \ge1 \}
\]
equipped with the relative topology inherited from the product topology on $\prod_{i=1}^{\infty} \, X_i$, $X_i=X$, $i=1,2, \dots$, and $s$ is the backward shift on $\S_{\phi}$.
\end{defn}

It is easy to see that in the case where $\fA= \mathrm{C}(X)$ and $\alpha$ is induced by an injective self-map $\phi$ of $X$, the pair $(\fB, \beta)$ for $(\fA, \alpha)$ described above, is conjugate to the solenoid system $(\S_{\phi}, s)$. Therefore, we obtain

\begin{cor}
Let $b$ be a non-constant finite Blaschke product, and let $\alpha(f) = f \circ b$ on $\CT$.  Then
\[
 \cenv(\AD \times_{\alpha} \bZ^+) =\cenv(\rC(\S_{b}) \times_{s} \bZ) .
\]
where $(\S_{b}, s)$ is the solenoid system of $(\bT , b)$.
\end{cor}

It is worth restating this theorem as a dilation result.

\begin{cor} \label{equivdil}
Let $\alpha$ be an endomorphism of $\AD$ induced by a non-constant
finite Blaschke product
and let $A,T \in \B(\H)$ be contractions satisfying $AT=T\alpha(A)$.
Then there exist unitary operators $U$ and $W$ on a Hilbert space
$\K \supset \H$ which simultaneously dilate $A$ and $T$, in the sense that
$P_\H W^m U^n |_\H = T^m A^n$ for all $m,n \ge 0$,
so that \mbox{$UW=W\alpha(U)$.}
\end{cor}

\begin{proof}
Every covariant representation $(A,T)$ of $(\AD,\alpha)$ dilates to a
covariant representation $(U_1,V)$ of $(\CT,\alpha)$.  This in turn dilates to
a maximal dilation $\tau$ of $\CT\times_\alpha\bZ^+$, in the sense of
Dritschel and McCullough \cite{DM}.  The maximal dilations extend to
$*$-representations of the C*-envelope.  Then $A$ is dilated to
$\tau(j(z)) = U$ is unitary and $T$ dilates to the unitary $W$ which implements the
automorphism $\beta$ on $\fB$, and restricts to the action of $\alpha$ on $\CT$.
\end{proof}

The situation changes when we move to \textit{non-injective} 
endomorphisms $\alpha$ of $\AD$.
Indeed, let $\lambda \in \bT$ and consider the endomorphism $\gal$ of $\AD$ 
induced by evaluation on $\lambda$, i.e., $\gal(f)(z)= f(\lambda)$, $\forall z \in \bD$. 
(Thus $\gal$ is the endomorphism of $\AD$ corresponding to a constant Blaschke product.) 
If two contractions $A,T$ satisfy $AT=T\gal(A) = \lambda T$, then the existence 
of unitary operators $U, W$, dilating $A$ and $T$ respectively, 
implies that $A=\lambda I$. 
It is easy to construct a pair $A,T$ satisfying $AT= \lambda T$ and yet $A\neq \lambda I$. 
This shows that the analogue Corollary \ref{equivdil} fails for $\alpha = \gal$ 
and therefore one does not expect $ \cenv(\AD \times_{\gal} \bZ^+) $ to be 
isomorphic to the crossed product of a commutative C*-algebra, 
at least under canonical identifications. 
However as we have seen, a weakening of Corollary \ref{equivdil} is valid 
for $\alpha = \gal$ if one allows $W$ to be an isometry instead of a unitary operator. 
In addition, we can identify $ \cenv(\AD \times_{\alpha} \bZ^+) $ as being 
\textit{strongly Morita equivalent} to a crossed product C*-algebra. 
Indeed, if
\[
 e \colon \bT \times \bN \longrightarrow \bT \times \bN
\]
is defined as
\[
e(z, n)=
\left\{
\begin{array}{cl} (1, 1) & \mbox{if } n =1 \\
                 (z, n-1) & \mbox{otherwise,}
                 \end{array}
                 \right.
                 \]
                 then
\begin{thm} \label{constantcase}
Let $\alpha= \gal$ be an endomorphism of $\AD$ induced by evaluation at a point $\lambda \in \bT$. Then
$\cenv(\AD \times_{\alpha} \bZ^+)$  is strongly Morita equivalent to $\rC(\S_{e}) \times_{s} \bZ$,
where $e \colon \bT \times \bN \longrightarrow \bT \times \bN$ is defined above and $(\S_{e}, s)$ is the solenoid system of $(\bT \times \bN , \, e)$.
\end{thm}

\begin{proof} In light of Lemma~\ref{envelopeidentified}, it suffices to identify the C*-envelope of $\CT \times_{\alpha} \bZ^+$. As $\alpha$ is no longer an injective endomorphism of $\CT$, we invoke the process of adding tails to C*-correspondences \cite{MT}, as modified in \cite{DR, KK2}.

Indeed, \cite[Example 4.3]{KK2} implies that the C*-envelope of the tensor algebra associated with the dynamical system $(\CT, \alpha)$ is strongly Morita equivalent to the Cuntz-Pimsner algebra associated with the injective dynamical system $(\bT \times \bN , \, e)$ defined above. Therefore by invoking the solenoid system of $(\bT \times \bN , \, e)$, the conclusion follows from the discussion following Lemma~\ref{envelopeidentified}.
\end{proof}


\end{document}